\documentclass{article}
\usepackage{amssymb, amsmath, amsthm}
\newtheorem{thm}{Theorem}

\newtheorem{lemma}[thm]{Lemma}

\newtheorem{claim}[thm]{Claim}
\newtheorem{subclaim}[thm]{Subclaim}
\DeclareSymbolFont{AMSb}{U}{msb}{m}{n}
\DeclareMathSymbol{\R}{\mathbin}{AMSb}{"52}

\newcommand{\beq}{\begin{equation}}
\newcommand{\eeq}{\end{equation}}
\newcommand{\beqa}{\begin{eqnarray}}
\newcommand{\eeqa}{\end{eqnarray}}

\author{Michael Bateman}
\title{ $L^p$ Estimates for Maximal Averages Along One-variable Vector Fields in ${\mathbf R} ^2$
\thanks{This work was supported in part under NSF Grant DMS0653763.
2000 Mathematics
Subject Classification:  Primary 42B25. }}
\vspace{2ex}
\author{Michael Bateman
\thanks{Department of Mathematics,
Indiana University,
Rawles Hall,
   831 East 3rd St,
   Bloomington, IN 47405({\tt mdbatema@indiana.edu.})}  }

\date{ }

\begin{document}

\maketitle

%\begin{abstract}
%We prove a conjecture of Lacey and Li in the case that the vector field depends only on one variable.  Specifically: 
% let $v$ be a vector field defined on the unit square such that $v(x,y) = (1,u(x))$ for some measurable 
%$u: [0,1] \rightarrow [0,1]$.  
%For a rectangle $R$, let $EX(R)$ be the interval centered at 
%let $V(R)= \{ (x,y) \in R \colon v(x,y) \in $
%Let $\mathcal R _{\delta} $ be the colection of rectangles such that ...
%\end{abstract}

\begin{abstract}
We prove a conjecture of Lacey and Li in the case that the vector field depends only on one variable.  Specifically: 
let $v$ be a vector field defined on the unit square such that $v(x,y) = (1,u(x))$ for some measurable 
$u: [0,1] \rightarrow [0,1]$.  Let $\delta$ be a small parameter, and let $\mathcal R$ be the collection of rectangles $R$ of a fixed width such that $\delta$ much of the vector field inside $R$ is pointed in (approximately) the same direction as $R$.  We show that the operator defined by
\beqa
M_{\mathcal R} f (z ) =   \sup _{z\in R \in \mathcal R}    {1 \over { |R| } } \int _{R} |f| 
\eeqa
is bounded on $L^p$ for $p>1$ with constants comparable to ${1 \over {\delta} }$.
\end{abstract}

 %For a rectangle $R$, let the interval of uncertainty $EX(R)$ be the interval of width ${{ W(R) }\over { L(R)} }$  centered at $slope(R)$.  Let 
%\beqa
%V(R) = \{ (x,y) \in R \colon v(x,y) \in EX(R)\}.
%\eeqa
%Fix $0<\delta \leq 1$, $0<w<<1$, and let 
%\beqa
%\widetilde{\mathcal R} _{\delta , w, v} = \{R \colon L(R) \leq {{1} \over {100||v||_{lip} } }, W(R) = w, \quad \text{and} \quad 
%|V(R)| \geq \delta |R| \}.  
%\eeqa
%We show that the operator defined by 
%\beqa
%M_{   \widetilde{\mathcal R} _{\delta , w, v}   } f (z) 
%= \sup _{z\in R \in   \widetilde{\mathcal R} _{\delta , w, v}   } {1 \over { |R| } } \int _{R} |f| .
%\eeqa
%is weak $(p,p)$ for $p>1$ with constants comparable to ${1 \over {\delta} }$.
%\end{abstract}

\section{Introduction}
In the paper [LL1], Lacey and Li reduce the problem of bounding in $L^2$ the Hilbert transform along a $C^{1+\varepsilon }$ vector field to  
estimating the $L^p$ norm of a related maximal function for some $p<2$.  They have established these maximal function bounds when 
$p=2$ and conjecture that they hold for $p>1$.  Here we prove the conjecture for vector fields of one variable.  More 
precise statements follow.

Let $v$ be a vector field on ${\mathbf R} ^2$.  We will assume $v \colon [0,1]\times [0,1] \rightarrow [0,1]$, i.e., we work only in a bounded region, and we assume all vectors are of the form $v(x,y) = (1 , u(x,y) )$.  
To define the maximal operator in question we need to introduce some notation.  For a rectangle $R$, we write $L(R)$ for its length, and $W(R)$ for its width.  Let $slope(R)$ be the slope of the long side of $R$.  (We will assume $L(R) \geq W(R)$.)  We define its interval of uncertainty $EX(R)$ to be the interval of width ${{ W(R) }\over { L(R)} }$  centered at $slope(R)$.  Let 
\beqa
V(R) = \{ (x,y) \in R \colon u(x,y) \in EX(R)\}.
\eeqa
Fix $0<\delta \leq 1$, $0<w<<1$, and let 
\beqa
\widetilde{\mathcal R} _{\delta , w, v} = \{R \colon L(R) \leq {{1} \over {100||v||_{lip} } }, W(R) = w, \quad \text{and} \quad 
|V(R)| \geq \delta |R|  \},
\eeqa
where 
$||v||_{lip}$ is the Lipschitz constant of the vector field $v$, and where $|\centerdot |$ indicates the Euclidean measure of a set.  In words:  $\tilde{\mathcal R} _{\delta , w, v}$ is the collection of rectangles $R$ such that $\delta$ much of the vector field in $R$ is pointed in (almost) the same direction as $R$.  

We will consider several similar maximal operators in this paper.  If $\mathcal R$ is a collection of rectangles, define 
\beqa
M_{\mathcal R} f (z) = \sup _{z\in R \in \mathcal R} {1 \over { |R| } } \int _{R} |f| .
\eeqa

%Finally, define
%\beqa
%\tilde{M} _{\delta , w, v} f(z) = \sup _{z\in R \in \tilde{\mathcal R} _{\delta , w, v} }  {1 \over { |R| } } \int _{R} |f| .
%\eeqa
%Lacey and Li proved the following:

Motivation for studying this operator comes from work of Lacey and Li [LL1], in which they prove Theorem \ref{hilbert}.
Define, for a sufficiently small value of $\beta$, the truncated integral operator
\beqa H_{v, \beta} f (z) = p.v. \int _{-\beta } ^ {\beta} { { f(z+ tv(z) ) }  \over { t} } dt. \eeqa
\begin{thm}[(Lacey-Li)] \label{hilbert}
Suppose there is a $p<2$ and an $N$ such that for any Lipschitz vector field $v$,
\beqa
|| M_{ \tilde{\mathcal R} _{\delta , w, v} }   || _{L^p \rightarrow L^p }\lesssim {1\over {\delta ^N} }
\eeqa
for any $0<w<<{{1} \over {100||v||_{lip} } }$.  Then if $v$ is a  $C^{1+\varepsilon }$ vector field, 
\beqa
||H_{v,\beta } || _{L^2 \rightarrow L^2 } \lesssim ( 1 + \log ||v|| _{ C^{1+\varepsilon } } ) ^2 .  
\eeqa
\end{thm}
It is interesting to note that this bound for $H_{v,\beta }$ is strong enough to prove Carleson's theorem on pointwise converge of Fourier series.  The reader is encouraged to consult [LL1] for the full story.
Here we prove that the hypothesis of this theorem is satisfied provided that the vector field $v$ depends only on one variable.  In fact, this additional assumption eliminates the need to assume that $v$ has any smoothness.  So now we define
\beqa
\mathcal R _{\delta , w, v}  = \{ R \colon |V(R)| \geq \delta |R| \quad \text{and} \quad L(R) \leq 1 \}.
\eeqa
%and
%\beqa
%M_{\delta , w, v} f(z) = \sup _{z \in R \in \mathcal R _{\delta , w, v} }  {1 \over { |R| } } \int _{R} |f| .
%\eeqa
\begin{thm} \label{main}
Let $v\colon [0,1] \times [0,1] \rightarrow \{1\} \times [0,1]$ depend only on the first variable, i.e., let $v(x,y) = (1,u(x)) $ for some measurable 
$u : [0,1] \rightarrow [0,1]$.  Then 
\beqa
||M _{\mathcal R _{\delta , w, v} }   || _{L^p \rightarrow L^p }\lesssim {1\over {\delta } },
\eeqa
with constants independent of $w$ and $v$.
\end{thm}
In section \ref{reductions}, we reduce the problem to a model with discrete slopes and paralellograms that project vertically to dyadic intervals.  There is essentially nothing new here, and experts may wish to skim for notation.  In section \ref{mainproof}, we prove Theorem \ref{main}.

\subsection{Acknowledgement}
The author thanks Nets Katz and Xiaochun Li for helpful discussions.

\section{Reductions} \label{reductions}
We begin by defining a discrete set of slopes.  Let
\beqa
S_k = \{ {  {j+ {1\over 2} } \over {2^k} } \colon j \in \{ 0,1,..., 2^k -1\}  \}.
\eeqa
Let 
\beqa
\mathcal R _{\delta , w, v} ^k = \{ R \in \mathcal R _{\delta , w, v} \colon 2^{k-1}w < L(R) \leq 2^k w  \text{  and  } slope(R) \in S_k \}.
\eeqa
Note that $R _{\delta , w, v} ^k $ is just a collection of rectangles in $R _{\delta , w, v}$ whose intervals of uncertainty have size about $2^{-k}$, and whose slopes are $2^{-k}$-separated.  Let
\beqa
R _{\delta , w, v} ^{dis} = \bigcup _{k=1} ^ {\log {1\over w} } R _{\delta , w, v} ^k.
\eeqa
Now we show that it is enough to consider averages over rectangles in $R _{\delta , w, v} ^{dis}$.  
\begin{lemma}
For any locally integrable function $f$, 
\beqa
M _{\mathcal R _{\delta , w, v} }f(z) \lesssim M _{\mathcal R^{dis} _{  { { \delta} \over {10} }  , 5w, v}      } f(z).
\eeqa
\end{lemma}
\begin{proof}
Let $R \in \mathcal R  _{\delta , w, v}$ with $|EX(R)| \sim 2^{-k} $.  There are two slopes $s_1$ and $s_2$ in $S_k$ such that 
\beqa
|s_j - slope(R) | \leq |EX(R)|.
\eeqa
There are (at least) two corresponding rectangles $R_1$ and $R_2$ such that $slope(R_j) = s_j$ and such that 
$R \subseteq 5 R_j$.  Further, either $|V(R_1) |  \geq { {\delta} \over {10} } |R_1|$ or $|V(R_2) |  \geq { {\delta} \over {10}} |R_2|$.  Say it holds for $R_1$.  Then 
\beqa
{1 \over { |R| } } \int _{R} |f| \leq { {25}  \over { |R_1| } } \int _{R_1} |f|,
\eeqa
and $R_1 \in R^{dis} _{  { { \delta} \over {10} }  , 5w, v} $.  This completes the proof.  
\end{proof}
Hence we may restrict our attention to the discrete model.  We will identify slopes with intervals.  That is, we will identify $s\in S_k$ with the dyadic interval centered at $s$.  So $S_0 = \{ [0,1] \}$, $S_1 = \{ [0,{1\over 2}), [ {1\over 2} , 1] \}$, 
$ S_2 = \{  [0,{1\over 4}), [ {1\over 4} , {1\over 2}],  [ {1\over 2} , {3\over 4}],  [ {3\over 4} , 1] \}$, etc.  With this identification, it is clear what we mean by $s\subseteq s'$ for $s \in S_k$ and $s' \in S_{k'}$ for some $k' <k$.  

We will further restrict our attention to a model in which we average over parallelograms that project vertically onto dyadic intervals.  The reduction to parallelograms is trivial.  Let $\mathcal D$ be the dyadic intervals, and let 
$\mathcal D '$ be the intervals in $\mathcal D$ shifted left by ${1\over 3}$.  It is not too difficult to check (use binary expansions) that if $J$ is an interval, then either there is a $K\in \mathcal D$ with $J\subseteq K$ and 
$|K| \leq 16 |J|$, or there is a $K\in \mathcal D'$ with $J\subseteq K$ and 
$|K| \leq 16 |J|$.  With this observation it is clear that we may control $M _{\mathcal R _{\delta , w, v} }$ with two dyadic models, with comparable values of the parameter $\delta$.

\section{Proof of main theorem} \label{mainproof}
We begin this section by rewriting the definition of the maximal operator under consideration, taking into account the reductions made in the previous section.  This will require some new notation.  Then we will state a covering lemma, and indicate how it yields Theorem \ref{main}.

\subsection{Notation}
Fix a small number $w$ (and for convenience, assume $w$ is an integer power of $2$).  
Let $u: [0,1] \rightarrow S_{\log {1\over w} } $, and let $v(x,y) = (1, u(x))$.  Now let $\mathcal D$ be the dyadic intervals contained in $[0,1]$.  Let $I\in \mathcal D$, and let 
$s \in S_{\log {  {|I| }  \over w} }$.  (Recall that parallelograms with length $|I|$ will only have slopes defined up to an error of ${ w \over { |I| } }$; this is why we take $s \in S_{\log {  {|I| }  \over w} }$.)  For the remainder of the paper, we will view $v$, $w$, and $\delta$ as being fixed.

Define the popularity of a slope $s$ in the interval $I$ to be 
\beqa
 Pop _I (s) = { 1 \over { |I| } }    | \{ x\in I \colon u(x) \subseteq s \} |.
\eeqa
(Again, recall that slopes are viewed as intervals, hence the notation $ u(x) \subseteq s$.)
%and $u$ takes values in a set with ${1 \over w}$ elements, corresponding to a very refined set of directions, and that     % rectangles smaller intervals $I$ have much coarser slopes.
Let 
\beqa
S(I) = \{ s\in  S_{\log {  {|I| }  \over w} } \colon Pop _I (s) \geq \delta \}.
\eeqa
This is the set of allowable slope for rectangles projecting to $I$.  
Given a parallelogram $R$, define $slope(R)$ to be the slope of the long side of $R$, and define $int(R)$ to be the projection of $R$ onto the $x$-axis.  We will let 
\beqa
\mathcal R = \{ \text{parallelograms  } R \colon  int(R) \in \mathcal D \text{  and  }  slope(R) \in S(I) \}.
\eeqa
Because of this, all intervals considered in the rest of the paper are assumed to be dyadic.
Recall that $M_{\mathcal R}$ is defined by 
\beqa
M_{\mathcal R} f (z) = \sup _{z\in R \in \mathcal R} {1 \over { |R| } } \int _{R} |f| 
\eeqa
for locally integrable $f$.  Our goal is to show $|| M_{\mathcal R} || _{L^p \rightarrow L^p } \lesssim {1 \over { \delta } }$.

\subsection{Statement of Covering Lemma}
We remark that there is nothing new about this covering lemma approach.  See, e.g., [LL2].
\begin{lemma} \label{covering}
Let $\widetilde{ \mathcal R } \subseteq \mathcal R$.  Let $q$ be an integer greater than or equal to $2$.  Then we may write $\widetilde{ \mathcal R }$ as the disjoint union of collections $ \mathcal R _1$ and $ \mathcal R _2 $ such that 
\beqa
\left| \bigcup _{R \in  \mathcal R _2 } R \right| \lesssim \sum _{R \in  \mathcal R _1}  |R|
\eeqa
and
\beqa \label{hardpart}
\int \left( \sum _{R\in \mathcal R _1 } \chi _ R \right)  ^ {q }     
 \leq c_q {1 \over { {\delta } ^ {q-1 } }  } \sum _{R \in  \mathcal R _1}  |R|.
\eeqa
\end{lemma}
To see that the lemma implies Theorem \ref{main}, let $f\in L^p$, let  ${1\over p} + {1\over q} =  1$, and let 
\beqa
E_{\lambda} = \{ M_{\mathcal R}f > \lambda \}.
\eeqa
Then write $E_{\lambda} = \bigcup _{R \in  \widetilde{ \mathcal R } } R $ for some 
$ \widetilde{ \mathcal R } \subseteq \mathcal R $, where 
\beqa
{1 \over { |R| } } \int _{R} |f| > \lambda
\eeqa
for $R \in  \widetilde{ \mathcal R }$ .  
We have a decomposition  $  \widetilde{ \mathcal R } = \mathcal R _1 \sqcup \mathcal R _2 $ as in the statement of the lemma, which gives us 
\beqa
\sum _{R \in  \mathcal R _1}  |R| & \leq & \sum _{R \in  \mathcal R _1}   {1 \over { \lambda } } \int _{R} |f|  \\ \nonumber
& \leq &{1 \over { \lambda } } || f || _p 
     \left( \int \left( \sum _{R\in  \mathcal R _1 } \chi _ R \right)  ^ {q } \right ) ^{ {1\over q} } \\ \nonumber
& \lesssim & {1 \over { \lambda } } || f || _p   
     {1 \over { {\delta } ^ {1-  {1\over q}  } }  }    \left( \sum _{R \in  \mathcal R _1}  |R| \right) ^{ {1\over q} }.
\eeqa
This implies
\beqa 
 \sum _{R \in  \mathcal R _1}  |R|  \lesssim { 1\over {\delta}  } {1 \over { \lambda ^ p} } ||f|| _p ^ p
\eeqa
This quantity obviously dominates $\left| \bigcup _{R \in  \mathcal R _1 } R \right|$, and it dominates 
$\left| \bigcup _{R \in  \mathcal R _2 } R \right| $ by the covering lemma.  Hence
\beqa
| E_ {\lambda} | \lesssim { 1\over {\delta}  } {1 \over { \lambda ^ p} } ||f|| _p ^ p .
\eeqa
This is the weak type $(p,p)$ estimate for $M_{\mathcal R}$.  Since we can prove the covering lemma for arbitrarily large integers $q$, we can prove weak type $(p,p)$ for any $p>1$.  %(Of course the implied constants in this estimate depend on $p$.)  

\subsection{Proof of Covering Lemma}

\subsubsection{Selection Procedure}

We construct the collections  $ \mathcal R _1$ and $ \mathcal R _2 $ as follows.  Initialize
\beqa
 \widetilde{ \mathcal R } & : = & \widetilde{ \mathcal R }  \\ \nonumber
 \mathcal R _1 & : = & \emptyset  \\ \nonumber 
 \mathcal R _2 & : = &  \emptyset .
\eeqa
While $ \widetilde{ \mathcal R }  \neq \emptyset$:  choose $R \in  \widetilde{ \mathcal R } $ of maximal length, and update
\beqa
 \widetilde{ \mathcal R } & : = & \widetilde{ \mathcal R }  \setminus  \{ R \}  \\ \nonumber
 \mathcal R _1 & : = &   \mathcal R _1  \cup \{ R \}  \\ \nonumber
 \mathcal R _2 & : = &  \mathcal R _2 ;
\eeqa
if there is an $R' \in \widetilde{ \mathcal R }$ such that $R' \subseteq \{ \sum _{R' \in \mathcal R _1 } \chi _{5R'}   \geq 1 \}$, update
\beqa
\widetilde{ \mathcal R } & : = & \widetilde{ \mathcal R }  \setminus  \{ R' \}  \\ \nonumber
 \mathcal R _1 & : = &   \mathcal R _1  \\ \nonumber
 \mathcal R _2 & : = &  \mathcal R _2 \cup \{ R' \} .
\eeqa
Here, of course, by $5R'$ we mean the parallelogram with the same center and side lengths inflated by a factor of $5$.
We make one important observation about the parallelograms in  $\mathcal R _1 $.  If $R,R' \in  \mathcal R _1 $ intersect, then they have different slopes.  More precisely, if $L(R) \leq L(R')$, then $slope(R) \not\supseteq slope(R')$.  For if it did, then 
$R \subseteq 5R'$.  (We are using the fact that $W(R)=W(R') $.)  Hence $R$ was put into the collection $\mathcal R _2$.  

It is clear by construction and by Chebyshev that 
\beqa
\left| \bigcup _{R \in  \mathcal R _2 } R \right| \lesssim \sum _{R \in  \mathcal R _1}  |R|,
\eeqa
so it remains to prove the estimate ( \ref{hardpart} ).  Note that 
\beqa
\int \left( \sum _{R\in \mathcal R _1 } \chi _ R \right)  ^ {q }    
\lesssim \sum _{R\in  \mathcal R _1 } \int _R  
\left(  \sum _{R' \in  \mathcal R _1, \text{  } int(R') \subseteq int(R) } \chi _{R'}    \right) ^{q-1},
\eeqa
so if we define 
\beqa
f(x,y) =  \sum _{R' \in  \mathcal R _1, \text{  }int(R') \subseteq int(R) } \chi _{R'} (x,y),
\eeqa
it is enough to show 
\beqa \label{uniform}
\int _R f(x,y) ^{q-1} \lesssim { 1 \over {\delta ^ {q-1} } } |R|
\eeqa
for any $R \in \mathcal R _1 $.  

\subsubsection{Uniform Estimates on Rectangles}

Without loss of generality, we will assume $int(R) = [0,1]$.  To prove $( \ref{uniform} )$, we will introduce some auxilliary functions.  To do this, we need some more notation.  The important point of this section is that we can control the two-variable function $f$ with a function of one variable that is relatively well-behaved. 

For $I \in \mathcal D$, define $T(I)$ as follows:  Let 
\beqa 
T([0,1]) = S([0,1]) . % =  \colon Pop_{ [0,1]} (s) \geq \delta \} .
\eeqa
Note that $T([0,1])$ is just the set of allowable slopes for the interval $[0,1]$.  (Recall that the allowable slopes for an interval are those that are at least $\delta$-popular.)  We will define $T(I)$ similarly, except that we will not include slopes that have been used by an ancestor of $I$.  More precisely, having defined 
$T(K)$ for $K\supsetneq  I$, define
\beqa
T(I)  =  \{ s \in S(I) \colon s\not\supseteq s' \text{  for any  } s' \in T(K), K\supsetneq I \} .
\eeqa
For $s \in T(I)$, let 
\beqa
 \mu _I ^s =|I| Pop _I (s)  = | \{ x\in I \colon u(x) \subseteq s \} |;
\eeqa
otherwise, let $\mu _I ^s = 0$.  
Now we define the  auxilliary functions:  let
\beqa
g(x) = \sum _I \chi _I (x) \# ( T(I) ).
\eeqa 
and
\beqa
h(x) = \sum _I \sum _{s\in T(I) } \chi _I (x) { {  \mu _I ^s } \over { |I| } } ,
\eeqa
Our strategy will be to control the function $f$ by the one-variable function $g$, and then to control $g$ by the function $h$, which we will show to be in ${\mathbf{BMO}} $.  
\begin{lemma} \label{control}
With $f$ and $h$ defined above, we have $f(x,y) \leq {1 \over {\delta } } h(x)$ for every $y$.  
\end{lemma}
\begin{lemma} \label{bmo}
With $h$ defined immediately above, $h \in {\mathbf{BMO}} _{dyadic} ([0,1]).$
\end{lemma}
With these two lemmas, we can easily finish the proof of Lemma \ref{covering}.
By the John-Nirenberg theorem, and the fact that $\int _{[0,1] } h (x) dx =1$, we have $||h||_r \leq c_r$ for any 
$1\leq r < \infty$.  Hence
\beqa
%\int _R \left( \sum  \chi _ R' \right) ^{q-1}  & = & 
\int _R \left( f(x,y) \right) ^{q-1} dxdy  
& \leq &  { 1 \over {\delta ^{q-1} } } w \int  _ { [0,1]} h(x) ^{q-1} dx  \\ \nonumber
& \leq & c_q  { 1 \over {\delta ^{q-1} } } |R|.
\eeqa
This completes the proof of the covering lemma.  We turn our attention to the proofs of Lemmas \ref{control} and 
\ref{bmo}.  Lemma \ref{bmo} is simple, and not really new, so we prove it first.

\begin{proof} [Proof of Lemma \ref{bmo} ]
Define $\mu _I = \sum _{s \in T(I) } \mu _I ^s $.  The sequence $\mu _I$ is a Carleson sequence; i.e., for any interval $I$, 
\beqa
\sum  _{J \subseteq I }   \mu _J \leq C |I|.
\eeqa
(In fact, we may take $C=1$ here.)  This holds because no $x$-coordinate can choose more than one slope.
Note that 
\beqa
h(x) = \sum _I \sum _{s\in T(I) } \chi _I (x) { {  \mu _I ^s } \over { |I| } } = \sum _I  \chi _I (x)  { {  \mu _I  } \over { |I| } }.
\eeqa
A function of this form is called a \it{balayage} \rm of the Carleson sequence $\mu _I$, and such functions are easily shown to be in 
 ${\mathbf{BMO}} _{dyadic}$.
To do this, it is enough to find, for each $I$, a number $b_I$ such that 
\beqa
{ 1 \over {|I|} } \int _I |h(x) - b_I | dx \leq C.
\eeqa
Let 
\beqa
b_I = \sum _{K \supseteq I }  { {  \mu _K  } \over { |K| } },
\eeqa
and compute, using the fact that $\mu _I$ is a Carleson sequence:
\beqa
{ 1 \over {|I|} } \int _I |h(x) - b_I | dx 
&=& { 1 \over {|I|} } \int _I  \sum _{J \subsetneq I } \chi _J (x) { {  \mu _J  } \over { |J| } } dx \\ \nonumber
&=& { 1 \over {|I|} } \sum _{J \subsetneq I } \mu _J  \\ \nonumber
& \leq & C.
\eeqa
\end{proof}
The proof of Lemma \ref{control} is a bit more involved.
\begin{proof}[Proof of Lemma \ref{control} ]
First note that if $s \in T(I)$, then ${ {  \mu _I ^s } \over { |I| } } \geq \delta $, so 
\beqa
g(x) \leq {1\over {\delta} } \sum _I \sum _{s\in T(I) } \chi _I (x) { {  \mu _I ^s } \over { |I| } }  = {1\over {\delta} } h(x).
\eeqa
it remains to show $f(x,y) \leq g(x)$ for all $y$.  

%Recall that each rectangle $R$ projects vertically to a dyadic interval.  Define $int(R)$ to be the vertical projection of $R$.  

Let 
\beqa
C(x,y) = \{ R \in \mathcal R _1  \colon (x,y) \in R\}, 
\eeqa
%\beqa
%D(x,y) = \{ I \colon \exists R\in C(x,y) \text{  with } int(R) = I \}.
%\eeqa
and let
\beqa
r_I(x,y) = \{ s \in S_{\log { { |I| } \over w} }  \colon \exists R \in C(x,y) \text{  with  }  int(R) = I \text{  and  } slope(R) = s \};
\eeqa
note that it may be empty for some $I$.  
Now we define two collections of pairs of intervals and slopes:
\beqa
P & = & \{ (I,s) \colon  s\in r_I (x,y) \} \\ \nonumber
Q & = & \{ (I,s) \colon  s\in T(I)   \} .
\eeqa
Two facts about the sets $P$ and $Q$ will finish the proof of Lemma \ref{control}.
\begin{claim}
We have

A.  $f(x,y) = \# (P)$ and $g(x) = \# (Q) $.

B.  $\# (P) \leq \# (Q)$.
\end{claim}
\begin{proof} [ Proof of Part A]
If $R \in C(x,y)$, then there is some $I$ with $int(R) = I$ and $s \in r_I (x,y) $ with $slope (R) =s$.  On the other hand, given 
$(I,s) \in P$, there can be at most one rectangle $R$ in $C(x,y)$ with $int(R) = I$ and 
 $slope (R) =s$.  For if there were two, then the shorter one would not be in the collection $\mathcal R _1$, by the construction of $\mathcal R_1$ and $\mathcal R_2$ .  The analogous fact for $g$ and $Q$ follows from the definitions of $g$, $Q$, and the collections $T(I)$.
\end{proof}
\begin{proof} [ Proof of Part B]
Of course it is enough to find an injection from $P$ to $Q$.  
\begin{subclaim}
Let $(I,s) \in P$.  Then there is $(J, t ) \in Q$ with $I\subseteq J$ and $s \supseteq t$.
\end{subclaim}
\begin{proof}
Note that $Pop _I (s) \geq \delta $, so if there is no $(J, t ) \in Q$ with $I \subsetneq J $ and $s \supseteq t$, then 
$(I,s) \in Q$ by definition of $T(I)$.  
\end{proof} % of subclaim 1
Now let $\alpha \colon P \rightarrow Q$ send $(I,s)$ to one of the elements $(J,t) \in Q$ provided by the subclaim.  (Such a choice may not be unique, but this is unimportant.)  The important point is this:
\begin{subclaim}
Suppose $(I_k,s_k) \in P$ and $(J_k ,t_k ) \in Q$ for $k=1,2$.  Suppose $I_k \subseteq J_k$ and $s_k \supseteq t_k$ for $k=1,2$.  If $(I_1,s_1) \neq (I_2,s_2)$, then $(J_1,t_1) \neq (J_2,t_2)$.
\end{subclaim}
Note that this subclaim guarantees that the function $\alpha$ above is one-to-one.
\begin{proof}
Let $(I_k,s_k)$ and $(J_k ,t_k ) $ be as in the statement of the subclaim with $(I_1,s_1) \neq (I_2,s_2)$.
We have both $x\in I_1$ and $x\in I_2$, so without loss of generality, we will assume $I_1 \subseteq I_2$.  
Since $s_1, s_2$ are dyadic intervals, we have that if $|s_1| \geq  |s_2|$, then either $s_1 \supseteq s_2$ or 
$s_1 \cap s_2 = \emptyset$.  We consider the following two cases:

CASE A:  $ s_1 \not\supseteq s_2$.  By the preceding observation, we have $s_1 \cap s_2 = \emptyset$.  Since 
$t_1 \subseteq s_1$ and $t_2 \subseteq s_2$, we have $t_1 \cap t_2 = \emptyset$, and hence $t_1 \neq t_2$.

CASE B:   $ s_1 \supseteq s_2$.  In fact, this case is not possible.  Suppose it were.  Then we would have $R_k$, 
$k = 1,2,$ with $int(R_k) = I_k$, $slope(R_k) = s_k$, and $R_1 \cap R_2 \neq \emptyset $.  But then 
$R_1 \subseteq 5R_2$, so $R_1 \not\in \mathcal R_1$.  

\end{proof} % of subclaim 2

\end{proof} % of claim part B

\end{proof}  % of lemma

\end{document}